\documentclass[dvipdfmx, 12pt, a4paper]{amsart}

\usepackage[utf8]{inputenc}
\usepackage[T1]{fontenc}
\usepackage{amsmath,amssymb,amsthm}
\usepackage{amsfonts}
\usepackage{graphicx}

\usepackage{tikz}
\usetikzlibrary{cd}

\theoremstyle{plain}
\newtheorem{thm}{Theorem}[section]
\newtheorem{prop}[thm]{Proposition}
\newtheorem{lem}[thm]{Lemma}

\newtheorem{claim}[thm]{Claim}

\theoremstyle{definition}

\newtheorem*{defi*}{Definition}

\newtheorem*{AI}{AI Disclosure}

\theoremstyle{remark}
\newtheorem{rmk}[thm]{Remark}
\newtheorem*{ack}{Acknowledgments}

\newcommand{\ra}{\rightarrow}

\newcommand{\cO}{\mathcal{O}}

\newcommand{\bC}{\mathbb{C}}

\newcommand{\bP}{\mathbb{P}}

\DeclareMathOperator{\Rat}{Rat}
\DeclareMathOperator{\ev}{ev}

\newcommand{\ol}{\overline}

\begin{document}
	
\title[Rational curves on cubic hypersurfaces]{Rational curves on cubic hypersurfaces in positive characteristic}
\author{Natsume Kitagawa}
\address{Graduate School of Mathematics, Nagoya University, Furocho Chikusa-ku, Nagoya, 464-8602, Japan}
\email{natsume.kitagawa.e6@math.nagoya-u.ac.jp}
\maketitle

\begin{abstract}
	We study the moduli spaces of rational curves on cubic hypersurfaces in characteristic $\neq2,3$. As a result, we prove that for every integer $d\geq1$ the Kontsevich moduli space of stable maps of degree $d$ on a smooth cubic hypersurface $X$ is irreducible if the dimension of $X$ is greater than or equal to $4$.
\end{abstract}

\section{Introduction}

	The study of rational curves on Fano varieties is an interesting subject. Rational curves play an important role in classification theory, such as the study of rational connectedness or the minimal model program.
	In this paper, we are interested in the irreducibility of the moduli spaces of rational curves on Fano varieties in positive characteristic.
	We study cubic hypersurfaces of dimension $\geq3$, defined over an algebraically closed field of characteristic $\neq2,3$. In the case that the dimension is greater than or equal to $4$, we obtain the following theorem:
	
	\begin{thm}\label{MainThm}
		Let $X$ be a smooth cubic hypersurface of dimension $n\geq4$ defined over an algebraically closed field of characteristic $\neq2,3$. Then for every integer $d\geq1$, the Kontsevich moduli space $\ol{M}_{0,0}(X,d)$ is irreducible and generically parametrizes free curves.
	\end{thm}
	
	This result extends \cite[Theorem 1.1]{CoskunStarr2009} which treats the case of characteristic $0$. Generalizing the theorem by Coskun and Starr to positive characteristic poses another difficulty: the existence of inseparable maps. Because of this, we cannot deduce the existence of free rational curves solely from the dominance of the evaluation maps.
	To overcome this difficulty, we use the dimension counting argument as in \cite{BeheshtiRiedl2021linear} that calculates the upper bound of the dimension of the non-free locus on $\ol{M}_{0,0}(X,d)$ inductively. The result in \cite{glas2025} on the dimension of fibres of evaluation maps enables us to
	carry out such an argument. Once the existence of free curves is proved, we can follow the argument in \cite[Proof of Theorem 1.1]{CoskunStarr2009} and we obtain the above result.
	
	By almost the same dimension counting argument, we also obtain the following result about cubic threefolds.
	
	\begin{thm}\label{MainThm3fold}
		Let $X$ be a smooth cubic threefold defined over an algebraically closed field of characteristic $\neq2,3$. Then for every integer $d\geq1$, the Kontsevich moduli space $\ol{M}_{0,0}(X,d)$ generically parametrizes free curves.
	\end{thm}
	
	In characteristic $0$, \emph{Geometric Manin's Conjecture}, proposed in \cite{LT2019GMC}, predicts the structure of moduli spaces of rational curves on Fano varieties; namely, the number of irreducible components of a given degree, the dimension of the irreducible components, and so on.
	
	The moduli spaces of rational curves on del Pezzo surfaces in positive characteristic are studied in \cite{BLRT2023Rat} from the viewpoint of Geometric Manin's Conjecture. Let $S$ be a del Pezzo surface, and $\ol\Rat(S)$ be the union of the irreducible components of $\ol{M}_{0,0}(S)$ which generically parametrize stable maps from irreducible domains.
	For a smooth cubic surface $S$, \cite[Theorem 1.1]{BLRT2023Rat} states that, if $S$ is not the Fermat cubic surface defined over a field of characteristic $2$, every dominant component of $\ol\Rat(S)$ is separable and generically parametrizes free rational curves. Theorem \ref{MainThm} and Theorem \ref{MainThm3fold} partially generalize this result to higher dimensional cases, and give new examples supporting Geometric Manin's Conjecture in positive characteristic.

	\begin{rmk}
		In \cite[Remark 1.4]{BLRT2023Rat} it is pointed out that the exceptions in \cite[Theorem 1.1]{BLRT2023Rat} coincide with non-globally $F$-regular del Pezzo surfaces if the degree of the del Pezzo surface is $\geq2$ (see also \cite[Section 3]{LT2026GMCinp}). We remark on global $F$-regularity of smooth cubic hypersurfaces and freeness of lines on non-globally $F$-regular ones.
		
		By \cite[Theorem B]{KawakamiTanaka2025dP}, a smooth cubic hypersurface is not globally $F$-regular only if the characteristic of the ground field is $2$. Combining Fedder's criterion and the argument in the proof of \cite[Theorem 4.4]{FFK2010GMRZ}, we see that a smooth cubic hypersurface which is not globally $F$-regular is projectively equivalent to the Fermat cubic hypersurface.
		
		Let $X$ be the Fermat cubic hypersurface of dimension $n\geq3$ defined over an algebraically closed field of characteristic $2$. Then any line on $X$ is non-free (see \cite[Proposition 1.2]{Furukawa2012cubic} for $n=3$ and \cite[Proof of Theorem 0.4]{FFK2010GMRZ} for $n\geq4$). In particular, Theorem \ref{MainThm} and Theorem \ref{MainThm3fold} do not extend to the Fermat cubic hypersurfaces in characteristic $2$.
	\end{rmk}

\subsection{Related works}

	In characteristic $0$, there are numerous results about moduli spaces of rational curves on Fano varieties. The work most relevant to the present paper is \cite{CoskunStarr2009}, which studied the space of rational curves on smooth cubic hypersurfaces of dimension $\geq4$ defined over $\bC$.
	\cite{Okamura2024} extended the results of \cite{CoskunStarr2009} to the case of del Pezzo varieties.
	
	Despite the progress in characteristic $0$, there are relatively few results in positive characteristic. The moduli spaces of rational curves on del Pezzo surfaces were studied in \cite{BLRT2023Rat}. For affine hypersurfaces of low degree, \cite{BrowningSawin2020geometric} proved the irreducibility of the moduli spaces using the circle method. \cite{glas2025} showed that, for smooth cubic hypersurfaces or smooth complete intersections of two quadrics of dimension $\geq3$, the moduli spaces of rational curves of a given degree on them have the expected dimensions. \cite{Manzuacteanu2021cubic} studied the moduli spaces of rational curves on smooth cubic hypersurfaces over finite fields. \cite{Beheshti2025separable} showed the separable rational connectedness of low degree Fano hypersurfaces by proving the existence of free lines with essentially sharp assumptions.
	
	Regarding Manin's Conjecture and Geometric Manin's Conjecture, \cite{Tanimoto2025Book} is a good reference. Also, \cite{LT2026GMCinp} provides a nice description of formulations and developments of the conjectures in positive characteristic.

\begin{ack}
	The author is grateful to his advisor, Professor Sho Tanimoto, for helpful discussions and continuous support. The author also thanks Professor Brian Lehmann, Professor Katsuhisa Furukawa and Fumiya Okamura for helpful discussions, answering his questions, careful reviews and constructive feedback. 
	%The author also would like to thank Shiho Oguihara. Thank you!!!
	
	This work was supported by JST SPRING, Grant Number JPMJSP2125.
\end{ack}

\begin{AI}
	Claude Opus 5.5 (Anthropic) was used to check typographical errors and polish the English. The author takes full responsibility for the content of this article.
\end{AI}

\section{Free lines}

	We work over an algebraically closed field $k$ of characteristic $\neq2,3$.

	\begin{lem}\label{DimofNonFreeLines}
		Let $X$ be a smooth cubic hypersurface of dimension $n\geq3$. Let $\Lambda_1\subset\ol{M}_{0,0}(X,1)$ be the locus parametrizing non-free lines. Then $\dim\Lambda_1\leq n-2$.
	\end{lem}
	
		\begin{proof}
			This is proved in \cite[Chapter 2, Lemma 2.12]{Huybrechts2023geometry}. We give the proof for the convenience of the reader.
			
			Fix a line $L$ on $X$. In some coordinates $x_0,\ldots,x_{n+1}$ on $\bP^{n+1}$, $L$ is written as $L=\{x_2=\cdots=x_{n+1}=0\}$. Consider the exact sequence of normal bundles:
			$$0\ra N_{L/X}\ra N_{L/\bP^{n+1}}\ra N_{X/\bP^{n+1}}|_L\ra0.$$
			This is isomorphic to
			$$0\ra N_{L/X}\ra\cO_L(1)^{\oplus n}\ra\cO_L(3)\ra0.$$
			Tensoring $\cO_L(-1)$, we have
			$$0\ra N_{L/X}(-1)\ra\cO_L^{\oplus n}\ra\cO_L(2)\ra0.$$
			Let $\alpha:\cO_L^{\oplus n}\ra\cO_L(2)$ be the surjective morphism in the above. Then by \cite[Corollary 6.16]{3264} and Euler's formula, $\alpha$ is written as
			$$\alpha=\left(\left.\frac{\partial F}{\partial x_i}\right|_L\right)_{i=2,\ldots,n+1}.$$
			Recall that $L$ is free (as a rational curve on $X$) if and only if $H^1(L,N_{L/X}(-1))=0$. This is equivalent to the condition that the induced map $\alpha:H^0(L,\cO_L^{\oplus n})\ra H^0(L,\cO_L(2))$ is surjective.
			
			Since $X$ is smooth, the rank of $\alpha:H^0(L,\cO_L^{\oplus n})\ra H^0(L,\cO_L(2))$ is at least $2$. Now consider the Gauss map $\gamma_X:X\ra X^*$ and the image $\gamma_X(L)$, where $X^*$ is the dual variety of $X$. If the rank of $\alpha$ is $3$ (i.e., $\alpha$ is surjective), the image $\gamma_X(L)$ is a smooth conic and $\gamma_X|_L:L\ra\gamma_X(L)$ is injective. If the rank of $\alpha$ is $2$, $\gamma_X|_L:L\ra\gamma_X(L)$ is a double covering of a line on $X^*$.
			
			Let $\widetilde{\Lambda}_1\subset\ol{M}_{0,1}(X,1)$ be the family of non-free lines. Since $\gamma_X:X\ra X^*$ is birational, the image $\ev(\widetilde{\Lambda}_1)\subset X$ is a proper closed subset, where $\ev:\ol{M}_{0,1}(X,1)\ra X$ is the evaluation map. Hence the dimension of $\ev(\widetilde{\Lambda}_1)$ is at most $n-1$. So it suffices to show that the restriction of the evaluation map $q:=\ev|_{\widetilde{\Lambda}_1}:\widetilde{\Lambda}_1\ra\ev(\widetilde{\Lambda}_1)\subset X$ is generically finite.
			
			For any $L\in\Lambda_1$, let $\iota_L:L\ra L$ be the covering involution for the restriction of the Gauss map $\gamma_X|_L:L\ra X^*$. Consider a map $\iota_x:q^{-1}(x)\ra X;L\mapsto\iota_L(x)$ for $x\in\ev(\widetilde{\Lambda}_1)$. Since $\gamma_X(\iota_x(L))=\gamma_X(x)$ for any $L\in q^{-1}(x)$, the image of $\iota_x$ is a finite set as the Gauss map is a finite morphism.
			
			For a line $L\in q^{-1}(x)$, if $x\neq\iota_L(x)$ then $L$ is the unique line through $x$ and $\iota_L(x)$. Therefore $\iota_x$ is injective on the subset of lines $L$ in $q^{-1}(x)$ satisfying $x\neq\iota_L(x)$, and thus this set is finite. Note that, fibrewise with respect to the forgetful map $\ol{M}_{0,1}(X,1)\ra\ol{M}_{0,0}(X,1)$, the set of points $(L,x)\in\widetilde{\Lambda}_1$ satisfying $x=\iota_L(x)$ is of codimension $1$. Thus, for a general point $x\in \ev(\widetilde{\Lambda}_1)$ the dimension of $q^{-1}(x)$ cannot be greater than $0$. Therefore $q$ is generically finite, and thus the dimension of $\widetilde{\Lambda}_1$ is at most $n-1$.
		\end{proof}

	\begin{prop}\label{LinesOnCubic}
		Let $X$ be a smooth cubic hypersurface of dimension $n\geq3$. Then the space $\ol{M}_{0,0}(X,1)$ of lines on $X$ is smooth, irreducible and generically parametrizes free lines.
	\end{prop}

		\begin{proof}
			The smoothness of $\ol{M}_{0,0}(X,1)$ follows from \cite[Corollary 1.12]{AltmanKleiman77}. However, in our situation we can prove it in an easier way; let $L$ be a line on $X$. By a normal bundle calculation, $N_{L/X}$ is isomorphic to either of the following:
			$$\cO_L^{\oplus2}\oplus\cO_L(1)^{\oplus n-3},\quad\text{or}\quad\cO_L(-1)\oplus\cO_L(1)^{\oplus n-2}.$$
			In either case, we obtain $H^1(L,N_{L/X})=0$. This implies that $\ol{M}_{0,0}(X,1)$ is smooth.
			
			By \cite[Theorem 1.16 (i)]{AltmanKleiman77}, $\ol{M}_{0,0}(X,1)$ is connected. Since it is smooth, we see that $\ol{M}_{0,0}(X,1)$ is irreducible.
			The last statement follows from Lemma \ref{DimofNonFreeLines} since the expected dimension of $\ol{M}_{0,0}(X,1)$ is equal to $2n-4>n-2$.
		\end{proof}

	\begin{lem}\label{IrredFibres}
		Let $X$ be a smooth cubic hypersurface of dimension $n\geq4$. Then a general fibre of the evaluation map $\ev:\ol{M}_{0,1}(X,1)\ra X$ is irreducible.
	\end{lem}

		\begin{proof}
			For a general point $x\in X$, the proof of \cite[Proposition 4.2]{glas2025} implies that $\ev^{-1}(x)$ is a complete intersection. Hence it is connected if $n\geq4$. By Lemma \ref{DimofNonFreeLines}, the union of non-free lines on $X$ is a proper closed subset. Therefore by \cite[3.5.4 Corollary, Chapter II]{Kollar1996rational}, $\ev^{-1}(x)$ is smooth. Therefore the general fibre is irreducible.
		\end{proof}

\section{Free curves and the irreducibility}

	We use the following result proved in \cite{glas2025}.

	\begin{prop}\label{DimOfFibres}
		Let $X$ be a smooth cubic hypersurface of dimension $n\geq3$. For an integer $d\geq1$, let $\ev_d:\ol{M}_{0,1}(X,d)\ra X$ be the evaluation map. Then there exists a finite set $K\subset X$ such that
		\begin{align*}
			\dim\ev_d^{-1}(x)=
			\begin{cases}
				(n-1)d-2\qquad\text{if}\quad x\notin K,\\
				(n-1)d-1\qquad\text{if}\quad x\in K.
			\end{cases}
		\end{align*}
	\end{prop}

		\begin{proof}
			This is proved in \cite[Proof of Theorem 1.3]{glas2025}.
		\end{proof}
	
	\begin{prop}\label{IrredDomain}
		Let $X$ be a smooth cubic hypersurface of dimension $n\geq4$ and let $K\subset X$ be a subset as in Proposition \ref{DimOfFibres}. If $p\notin K$, the fibre $\ev_d^{-1}(p)$ generically parametrizes curves with irreducible domains.
	\end{prop}

		\begin{proof}
			Using Proposition \ref{DimOfFibres}, the proof of \cite[Proposition 2.5]{CoskunStarr2009} works.
		\end{proof}

	\begin{lem}\label{DimofNonFreeCurvesonCubic}
		Let $X$ be a smooth cubic hypersurface of dimension $n\geq4$. Then for any positive integer $d$, the locus in $\ol{M}_{0,0}(X,d)$ parametrizing stable maps with at least one non-free component has dimension at most $d(n-1)$. In particular, every component of $\ol{M}_{0,0}(X,d)$ generically parametrizes free curves.
	\end{lem}

		\begin{proof}
			We mimic the proof of \cite[Theorem 5.1]{BeheshtiRiedl2021linear}.
			
			The proof proceeds by induction on the degree $d$. The case of $d=1$ follows from Lemma \ref{DimofNonFreeLines}. 
			Suppose that $d\geq2$ and the assertion holds for every degree less than $d$. Let $\Lambda_d\subset\ol{M}_{0,0}(X,d)$ be a closed subscheme generically parametrizing stable maps with at least one non-free component. If $\dim\Lambda_d\leq2n-2$, there is nothing to prove. Hence we assume $\dim\Lambda_d>2n-2$. Let $Y$ be the subvariety of $X$ swept out by the image of maps parametrized by $\Lambda_d$. Then $\dim Y\leq n$, and thus there are two distinct points of $Y$ and an at least $1$-dimensional family of stable maps parametrized by $\Lambda_d$ through these two points. Therefore we can apply the bend and break lemma; we see that the locus $\Lambda'_d\subset\Lambda_d$ parametrizing stable maps with reducible domain has codimension at most $1$ in $\Lambda_d$, and thus there exists a positive integer $e<d$ such that the locus of stable maps which decompose as a degree $e$ stable map with at least one non-free component glued at a point to a degree $d-e$ stable map has dimension at least $\dim\Lambda_d-1$.
			
			By the induction hypothesis and Proposition \ref{DimOfFibres}, the locus $\Lambda'_d$ of glued maps has dimension at most
			$$\{e(n-1)\}+\{(d-e)(n-1)-2\}+1=d(n-1)-1,$$
			and thus
			$$\dim\Lambda_d\leq\dim\Lambda'_d+1\leq d(n-1).$$
			So we obtain the desired result for $d$.
		\end{proof}

	\begin{proof}[Proof of Theorem 1.1]
		The proof is modeled on \cite[Proof of Theorem 1.1]{CoskunStarr2009}.
		
		The case of $d=1$ follows from Proposition \ref{LinesOnCubic}. Let $\ev_d:\ol{M}_{0,1}(X,d)\ra X$ be the evaluation map. By Lemma \ref{DimofNonFreeCurvesonCubic}, every irreducible component of $\ol{M}_{0,0}(X,d)$ generically parametrizes free curves. In particular, the restriction of the evaluation map to each component is dominant. Thus it suffices to show the irreducibility of the general fibre of $\ev_d$. We prove it by induction on $d$. The case of $d=1$ is proved in Lemma \ref{IrredFibres}. Suppose $d\geq2$ and the assertion holds for $e<d$. Let $\Delta_{\{e,\{1\}\},\{d-e,\emptyset\}}\subset\ol{M}_{0,1}(X,d)$ be the locus parametrizing stable maps that decompose as a degree $e$ stable map with a marked point and a degree $d-e$ stable map glued at a point. Take a general point $x\in X$ and let $\pi:\ol{M}_{0,2}(X,e)\ra\ol{M}_{0,1}(X,e)$ be the forgetful morphism forgetting the second marked point. We show the following claim:
		\begin{claim}\label{claim1}
			The intersection $\ev_d^{-1}(x)\cap\Delta_{\{e,\{1\}\},\{d-e,\emptyset\}}$ is irreducible and of codimension $1$ in $\ev_d^{-1}(x)$.
		\end{claim}
		\begin{proof}
			By Proposition \ref{IrredDomain}, the general map in $\ev_e^{-1}(x)$ has an irreducible domain. Hence $\pi^{-1}(\ev_e^{-1}(x))$ is irreducible, since $\ev_e^{-1}(x)$ is irreducible by the induction hypothesis.
			
			The general fibre of the first projection
			$$\pi_1:\pi^{-1}(\ev_e^{-1}(x))\times_X\ol{M}_{0,1}(X,d-e)\ra\pi^{-1}(\ev_e^{-1}(x))$$
			is irreducible by the induction hypothesis. By Proposition \ref{DimOfFibres}, we conclude that $\pi^{-1}(\ev_e^{-1}(x))\times_X\ol{M}_{0,1}(X,d-e)$ is irreducible and of dimension $d(n-1)-3$.
			
			Consider the restriction of the glueing map to $\pi^{-1}(\ev_e^{-1}(x))\times_X\ol{M}_{0,1}(X,d-e)$. Note that its image coincides with $\ev_d^{-1}(x)\cap\Delta_{\{e,\{1\}\},\{d-e,\emptyset\}}$. Since the domain is irreducible and has dimension $d(n-1)-3$, so is the image.
		\end{proof}
		We now return to the proof of Theorem \ref{MainThm}. Let $M$ be an irreducible component of $\ev_d^{-1}(x)$ for a general point $x\in X$. By the bend and break lemma, $M$ must contain stable maps with reducible domains in codimension $1$. By Claim \ref{claim1}, $M$ must contain the entire intersection of at least one boundary component $\Delta_{\{e,\{1\}\},\{d-e,\emptyset\}}$ with $\ev_d^{-1}(x)$ for some $e<d$. This locus must contain all stable maps $f:C=C_1\cup\dots\cup C_d\ra X$ obtained by glueing $d$ lines with a marked point $p$ (mapped to $x$) contained in the first curve $C_1$. Therefore we see that every component must contain such stable maps. If we choose a general $f$, then the restriction of the stable map to each component is free. Hence we have $H^1(C,f^*T_X(-p))=0$ and thus $\ev_d^{-1}(x)$ is smooth at $[f]$. This shows that $\ev_d^{-1}(x)$ is irreducible. Hence $\ol{M}_{0,0}(X,d)$ is irreducible.
	\end{proof}

\section{Free curves on cubic threefolds}

	Let $X$ be a smooth cubic threefold defined over an algebraically closed field $k$ of characteristic $\neq2,3$. By the same argument as in Lemma \ref{DimofNonFreeCurvesonCubic}, we can prove that $\ol{M}_{0,0}(X,d)$ generically parametrizes free rational curves for any $d\geq1$. However, for dimension reasons, we have to check the case of the space of conics $\ol{M}_{0,0}(X,2)$ directly.

	\begin{lem}\label{Conics}
		The space $\ol{M}_{0,0}(X,2)$ consists of two irreducible components $R_2$, $N_2$, where $R_2$ generically parametrizes birational stable maps from irreducible curves to smooth conics and $N_2$ parametrizes degree $2$ covers from $\bP^1$ to lines. Both of them generically parametrize free curves.
	\end{lem}
	
		\begin{proof}
			By \cite[Proof of Proposition 7.4]{LT2019GMC}, $\ol{M}_{0,0}(X,2)$ consists of two components $R_2$, $N_2$. Clearly $N_2$ generically parametrizes free curves. Glueing two free lines, we see that $R_2$ generically parametrizes free curves.
		\end{proof}

	\begin{proof}[Proof of Theorem 1.2]
		As in Lemma \ref{DimofNonFreeCurvesonCubic}, it suffices to show that the locus in $\ol{M}_{0,0}(X,d)$ parametrizing stable maps with at least one non-free component has dimension at most $2d-1$.
		The base case of induction is $d=2$ because we cannot use the bend and break lemma for the space of conics for dimension reasons. By Lemma \ref{Conics}, each component of $\ol{M}_{0,0}(X,2)$ generically parametrizes free curves. Hence the locus parametrizing stable maps with at least one non-free component in $\ol{M}_{0,0}(X,2)$ has dimension at most the expected dimension minus $1$, which is equal to $3=2\cdot(3-1)-1$, thus the assertion holds. The rest is the same as the proof of Lemma \ref{DimofNonFreeCurvesonCubic}.
	\end{proof}

\bibliographystyle{amsalpha}
\bibliography{cubic}

\end{document}